\documentclass{article}
\usepackage{amssymb}
\usepackage{amsmath}
\usepackage{amsthm}
\usepackage{amsfonts}
\begin{document}
\def\mathscr{\mathcal}
\topmargin= -.2in \baselineskip=15pt
\newtheorem{theorem}{Theorem}[section]
\newtheorem{proposition}[theorem]{Proposition}
\newtheorem{lemma}[theorem]{Lemma}
\newtheorem{corollary}[theorem]{Corollary}
\newtheorem{conjecture}[theorem]{Conjecture}
\theoremstyle{remark}
\newtheorem{remark}[theorem]{Remark}

\title {Deformations and Rigidity of $\ell$-adic Sheaves
\thanks{I would like to thank P. Deligne, Y. Hu, N. Katz, W. Zheng and J. Xie for
their precious suggestions and comments. 
P. Deligne proves Lemma \ref{conj}. N. Katz clarifies 
the notion of rigidity used in this paper. J. Xie proves
Lemma \ref{xie}. I am also thankful for 
H. Esnault who invited me to visit Essen in 2005 
and suggested the problem studied in this paper. This research is supported by the NSFC.}}

\author {Lei Fu\\
{\small Yau Mathematical Sciences Center, Tsinghua University,
Beijing, China}\\
{\small leifu@mail.tsinghua.edu.cn}}

\date{}
\maketitle

\begin{abstract}
Let $X$ be a smooth connected projective curve over an algebraically
closed field,  and let $S$ be a finite nonempty closed subset in $X$. 
We study 
deformations of lisse $\overline{\mathbb F}_\ell$-sheaves on $X-S$. The universal deformation space is a
formal scheme. Its generic fiber has a rigid analytic space
structure. By a dimension counting of this rigid analytic space, we prove a
conjecture of Katz which says that a lisse irreducible $\overline{\mathbb
Q}_\ell$-sheaf $\mathcal F$ on $X-S$ is rigid if and only if  
$\mathrm{dim}\, H^1(X,j_\ast\mathcal End(\mathcal F))=2g$, where $j:X-S\to X$
is the open immersion, and $g$ is the genus of $X$. 

\noindent {\bf Key words:} deformation of Galois representations,
rigid analytic space, $\ell$-adic sheaf. 

\noindent {\bf Mathematics Subject Classification:} 14D15, 14G22.

\end{abstract}

\section*{Introduction}

In this paper, we work over an algebraically closed ground field $k$ of arbitrary
characteristic. Let $X$ be a smooth connected
projective curve over $k$, let $S$ be a nonempty finite closed subset of $X$,
and let $\ell$ be a prime number distinct from the characteristic of $k$. For any $s\in
S$, let $\eta_s$ be the generic point of the strict henselization of
$X$ at $s$. A lisse $\overline{\mathbb Q}_\ell$-sheaf $\mathcal F$
on $X-S$ is called {\it physically rigid} if for any lisse
$\overline{\mathbb Q}_\ell$-sheaf $\mathcal G$ on $X-S$ with the
property $\mathcal F|_{\eta_s}\cong \mathcal G|_{\eta_s}$ for all
$s\in S$, we have $\mathcal F\cong \mathcal G$. 
To get a good notion of physical rigidity, we have to
assume $X=\mathbb P^1$. Indeed, the abelian pro-$\ell$ quotient of the
\'etale fundamental group 
$\pi_1(X)$ of $X$ is isomorphic to $\mathbb Z_\ell^{2g}$, where $g$ is the genus 
of $X$. If $g\geq 1$, then 
there exists a character $\chi:\pi_1(X)\to \overline{\mathbb Q}_\ell^\ast$ such 
that $\chi^n$ are nontrivial for all $n$. So 
there exists a lisse $\overline{\mathbb Q}_\ell$-sheaf $\mathcal L$
of rank $1$ on $X$ such that $\mathcal L^{\otimes n}$ are nontrivial for all $n$. For any lisse $\overline{\mathbb Q}_\ell$-sheaf
$\mathcal F$ on $X-S$, the lisse sheaf $\mathcal G=\mathcal
F\otimes(\mathcal  L|_{X-S})$ is not isomorphic to $\mathcal F$ since they
have non-isomorphic determinant, but $\mathcal
F|_{\eta_s}\cong\mathcal G|_{\eta_s}$ for each $s\in X$ since $\mathcal L$ is lisse at $s$ and hence 
$\mathcal L|_{\eta_s}$ is trivial. Thus
$\mathcal F$ is not physically rigid. 
So we modify the concept of physical rigidity as follows. (Confer \cite[1.2.1]{K}). 
A lisse $\overline{\mathbb Q}_\ell$-sheaf $\mathcal F$ on $X-S$ is called \emph{rigid} if 
there exists a finite family $\mathcal F^{(j)}$ $(j\in J)$ of lisse
$\overline{\mathbb Q}_\ell$-sheaves on $X-S$ such that for any lisse $\overline{\mathbb Q}_\ell$-sheaf $\mathcal G$ on 
$X-S$ with the property  $$\mathrm{det}(\mathcal F)\cong\mathrm{det}(\mathcal G),\quad
\mathcal F|_{\eta_s}\cong \mathcal G|_{\eta_s}\hbox{ for all }
s\in S,$$ we have $\mathcal G\cong \mathcal F^{(j)}$ for some $j\in J$. 

In \cite[5.0.2]{K}, Katz shows that for an irreducible lisse $\overline{\mathbb Q}_\ell$-sheaf $\mathcal F$ on $\mathbb P^1-S$, if 
$H^1(\mathbb P^1, j_{\ast}\mathcal End(\mathcal F))=0,$ where 
$j:\mathbb P^1-S\hookrightarrow \mathbb P^1$ is the open immersion, then
$\mathcal F$ is physically rigid. Moreover, he proves (\cite[1.1.2]{K}) that if $k=\mathbb C$ is the 
complex number field and $\mathcal F$ is an irreducible physically rigid complex local system on 
$\mathbb P^1-S$, then we have $H^1(\mathbb P^1, j_{\ast}\mathcal End(\mathcal F))=0.$ 
In this paper, we work over
any field $k$ and any curve $X$. We prove the following, which is analogous to \cite[1.2.3]{K} in the complex 
ground field case.

\begin{theorem} \label{maintheorem} Let $X$ be a smooth projective curve over an algebraically
closed field $k$, let $S$ be a nonempty finite closed subset of $X$, and let $\mathcal F$
be a lisse $\overline{\mathbb Q}_\ell$-sheaf on $X-S$. Suppose 
that $\mathcal F$ is irreducible and rigid.
Then $$\mathrm{dim}\, H^1(X,j_\ast \mathcal End(\mathcal F))= 2g \hbox{ and }
H^1(X,j_\ast \mathcal End^{(0)}(\mathcal F))= 0,$$
where $j:X-S\hookrightarrow X$ is the 
open immersion, $\mathcal End^{(0)}(\mathcal F)$ is the subsheaf of $\mathcal End(\mathcal F)$ 
of trace $0$, and $g$ is the genus of $X$. 
\end{theorem}

Katz's work for the complex field case can be
interpreted as a study of the moduli space of representations of the
topological fundamental group of $\mathbb P^1-S$. In \cite[Theorem 4.10]{BE},
Bloch and Esnault study deformations of locally free $\mathcal
O_{\mathbb P^1-S}$-modules provided with connections while keeping local
(formal) data undeformed, and prove that the universal deformation
space is algebraizable. Using this fact, they obtain a cohomological criterion 
for  physical rigidity for 
irreducible locally free
$\mathcal O_{\mathbb P^1-S}$-modules provided with connections. Our method is
similar. The moduli space of $\ell$-adic sheaves does not exist in the usual sense. 
We study deformations of lisse $\overline{\mathbb F}_\ell$-sheaves.
The universal deformation space is a formal scheme, and its generic
fiber is a rigid analytic space which can be used to produce
families of $\overline{\mathbb Q}_\ell$-sheaves. We use this rigid analytic space 
as a substitute for the moduli space of $\overline{\mathbb Q}_\ell$-sheaves. By a counting
argument on dimensions of rigid analytic spaces, we solve Katz's problem. 

We prove Theorem \ref{maintheorem} in \S 1. Let's prove the following corollary. 

\begin{corollary} Suppose that $\mathcal F$ is an irreducible rigid lisse $\overline{\mathbb Q}_\ell$-sheaf on $X-S$.

(i) In the case $g=0$, $\mathcal F$ is physically rigid if and only if $\mathrm{dim}\, H^1(X,j_\ast \mathcal End(\mathcal F))=0$.

(ii) In the case $g\geq 2$, $\mathcal F$ is rigid if and only if $\mathcal F$ is of rank 1. 

(iii) In the case $g=1$, the following conditions are equivalent:

\quad (1) $\mathcal F$ is rigid.

\quad (2)  $\mathrm{dim}\, H^1(X,j_\ast \mathcal End(\mathcal F))=2$.

\quad (3) If $\mathcal G$ is a lisse $\overline{\mathbb Q}_\ell$-sheaf on $X-S$ such that $\mathcal F|_{\eta_s}\cong\mathcal G|_{\eta_s}$ 
for all $s\in G$, then there a rank one $\overline{\mathbb Q}_\ell$-sheaf $\mathcal L$ lisse everywhere $X$ such that $\mathcal G\cong
\mathcal F\otimes j^*\mathcal L$. 
\end{corollary}

\begin{proof}   (i) follows directly from Theorem \ref{maintheorem} and 
\cite[5.0.2]{K}. 

(ii) Using the definition of rigidity, one verifies directly that any rank 1 sheaf is rigid. 
Suppose $\mathcal F$ is rigid of rank $r$. By Theorem
\ref{maintheorem}, we have $\mathrm{dim}\, H^1(X, j_*\mathcal End(\mathcal F))=2g$. 
Since $\mathcal F$ is irreducible, the dimension of 
$H^0(X, j_*\mathcal End(\mathcal F))\cong \mathrm{End}(\mathcal F)$
is $1$ by Schur's lemma. Then by the Poincar\'e duality (\cite[1.3 and 2.2]{Ddualite}), we have 
$\mathrm{dim}\, H^2(X, j_*\mathcal End(\mathcal F))=1.$ So we have 
$
\chi(X,  j_*\mathcal End(\mathcal F))= 
2-2g.$
Combined with the Grothendieck-Ogg-Shafarevich formula, we get
\begin{eqnarray*}
(2-2g)(1-r^2)&=& \chi(X,  j_*\mathcal End(\mathcal F))-(2-2g)r^2\\
&=&-\sum_{s\in S}\Big(r^2-\mathrm{dim}(j_*\mathcal End(\mathcal F))_{\bar s}
+\mathrm{sw}_s(\mathcal End(\mathcal F))\Big),
\end{eqnarray*} where $\mathrm{sw}_s$ is the Swan conductor at $s$. 
We have $r^2-\mathrm{dim}\, (j_*\mathcal End(\mathcal F))_{\bar s}\geq 0$ and $\mathrm{sw}_s(\mathcal End(\mathcal F))\geq 0$. 
It follows that $(2-2g)(1-r^2)\leq 0$. If $g\geq 2$, this last inequality implies that $r=1$.

(iii)  (1)$\Rightarrow$(2) Follows from Theorem
\ref{maintheorem} for the case $g=1$. 

(2)$\Rightarrow$(3) The calculation in the proof of (ii)  shows that in the case $g=1$,  
we have $$-\sum_{s\in S}\Big(r^2-\mathrm{dim}\, (j_*\mathcal End(\mathcal F))_{\bar s}
+\mathrm{sw}_s(\mathcal End(\mathcal F))\Big)=0.$$
But we have $r^2-\mathrm{dim}\, (j_*\mathcal End(\mathcal F))_{\bar s}\geq 0$
and $\mathrm{sw}_s(\mathcal End(\mathcal F))\geq 0$. So the equalities must hold. 
This shows that $j_*\mathcal End(\mathcal F)$ must be lisse at each $s\in S$. We then use the same the proof of 
(6)$\Rightarrow$(1) of \cite[1.3.1]{K}

(3)$\Rightarrow$(1) Let $\mathcal G$ be a lisse $\overline{\mathbb Q}_\ell$-sheaf on $X-S$ such that 
$\mathrm{det}(\mathcal F)\cong \mathrm{det}(\mathcal G)$ and $\mathcal F|_{\eta_s}\cong\mathcal G|_{\eta_s}$ 
for all $s\in S$. By condition (iii), we have $\mathcal G\cong
\mathcal F\otimes j^*\mathcal L$ for some rank one $\overline{\mathbb Q}_\ell$-sheaf $\mathcal L$ lisse everywhere $X$.
Since $\mathrm{det}(\mathcal F)\cong \mathrm{det}(\mathcal G)$, we must have $\mathcal L^{\otimes r}\cong \overline{\mathbb Q}_\ell$. 
Let's prove that there are only finitely many such sheaves $\mathcal L$ (up to isomorphism). 
It suffice to show that there are only finitely many characters $\chi:\pi_1(X,\bar\eta)\to\
\overline{\mathbb Q}_\ell^\ast$ with the property $\chi^r=1$. Indeed, if $\mu_r$ is the group of 
$r$-th roots of unity in $\overline{\mathbb Q}_\ell$, we have 
$\mathrm{Hom}(\pi_1(X,\bar\eta), \mu_r)\cong H^1(X,\mu_r)$
by \cite[XI 5]{SGA1}, and $H^1(X,\mu_r)$ is finite by \cite[XIV 1.2]{SGA4}.
\end{proof}

\section{Proof of Theorem \ref{maintheorem}}

In this section, we prove Theorem \ref{maintheorem} except for the technical Lemmas 
\ref{fiber},  \ref{conj} and \ref{xie}.
Their proofs are left to later sections. 

In the following, we take $\Lambda$ to be either a finite extension
$E$ of $\mathbb Q_\ell$, or the integer ring $\mathcal O$ of such
$E$. Let $\mathfrak
m$ be the maximal ideal of $\Lambda$, and let
$\kappa_\Lambda=\Lambda/\mathfrak m$ be the residue field of $\Lambda$. Denote by $\mathcal C_\Lambda$ the
category of Artinian local $\Lambda$-algebras with same residue
field $\kappa_\Lambda$ as $\Lambda$. Morphisms in $\mathcal C_{\Lambda}$ are
homomorphisms of $\Lambda$-algebras. If $A$ is an object in $\mathcal
C_\Lambda$, we denote by $\mathfrak m_A$ the maximal ideal of $A$.
Let $\eta$ be the generic point of $X$, and let $\bar\eta$ and $\bar\eta_s$ be geometric points over 
$\eta$ and $\eta_s$, respectively. Choose morphisms $\bar\eta_s\to \bar\eta$ so that the diagrams
$$\begin{array}{ccc}
\bar\eta_s&\to&\bar\eta\\
\downarrow&&\downarrow\\
\eta_s&\to&\eta
\end{array}$$ 
commute for all $s\in S$. 
They induce  canonical homomorphisms $\mathrm{Gal}(\bar\eta_s/\eta_s)\to \pi_1(X-S, \bar\eta)$, where
$\pi_1(X-S,\bar\eta)$ is the \'etale fundamental group of $X-S$. A homomorphism
$\rho:\pi_1(X-S,\bar\eta)\to \mathrm{GL}(A^r)$ is called a
representation if it is continuous. Here, if $\Lambda=\mathcal O$, then $A$ is finite and we put the discrete topology on
$\mathrm{GL}(A^r)$. If $\Lambda=E$, then $A$ is a finite dimensional
vector space over $E$, and we endow $\mathrm{GL}(A^r)$ with the topology induced from the
$\ell$-adic topology on $A$. Denote by $\rho|_{\mathrm{Gal}(\bar\eta_s/\eta_s)}$ the composite 
$$\mathrm{Gal}(\bar\eta_s/\eta_s)\to \pi_1(X-S, \bar\eta)\stackrel{\rho}\to \mathrm{GL}(A^r).$$ 

Suppose we are given a (continuous) representation
$\rho_\Lambda:\pi_1(X-S,\bar\eta)\to \mathrm{GL}(\Lambda^r)$. Let
$\rho_0:\pi_1(X-S,\bar\eta)\to \mathrm{GL}(\kappa_\Lambda^r)$ be the
representation obtained from $\rho_\Lambda$ by passing to residue
field. (If $\Lambda=E$, then $\rho_0$ coincides with $\rho_\Lambda$.)
We study deformations of $\rho_0$. Our treatment is similar
to Mazur's theory of deformations of Galois representations
(\cite{M}) and Kisin's theory of framed deformations of Galois
representations (\cite{Kisin}).
Suppose we
are given $P_{0,s}\in\mathrm{GL}(\kappa_\Lambda^r)$ for each $s\in S$ with the property
$$P_{0,s}^{-1}\rho_0|_{\mathrm{Gal}(\bar\eta_s/\eta_s)}P_{0,s}=\rho_0|_{\mathrm{Gal}(\bar\eta_s/\eta_s)}.$$ 
In application, we often take $P_{0,s}$ to be the identity matrix $I$
for all $s\in S$. In this case, we denote the data $(\rho_0,
(P_{0,s})_{s\in S})$ by $(\rho_0, (I)_{s\in S})$. For any
$A\in\mathrm{ob}\,\mathcal C_\Lambda$, denote the composite
$$\pi_1(X-S,\bar\eta)\stackrel{\rho_\Lambda}\to\mathrm{GL}(\Lambda^r) 
\to\mathrm{GL}(A^r)$$ also by $\rho_\Lambda$. Let $\lambda=\mathrm{det}(\rho_\Lambda)$.
Define $F^\lambda(A)$ to be the set of 
deformations $(\rho,(P_s)_{s\in S})$ of the data $(\rho_0,
(P_{0,s})_{s\in S})$ with $\mathrm{det}(\rho)=\lambda$ and with the prescribed local monodromy 
$\rho_{\Lambda}|_{\mathrm{Gal}(\bar\eta_s/\eta_s)}$. More precisely, we define it to be the set
of equivalent classes
\begin{eqnarray*}
F^\lambda(A)=\{(\rho,(P_s)_{s\in S})&|&\rho:\pi_1(X-S,
\bar\eta)\to\mathrm{GL}(A^r) \hbox{ is a continuous representation},\; P_s\in \mathrm{GL}(A^r), \\
&&\rho\mod\mathfrak m_A=\rho_0, \quad P_s \mod \mathfrak m_A=P_{0,s},\\
&&\mathrm{det}(\rho)=\lambda, \quad P_s^{-1}\rho|_{\mathrm{Gal}(\bar\eta_s/\eta_s)} P_s=\rho_\Lambda|_{\mathrm{Gal}
(\bar\eta_s/\eta_s)}\hbox{ for all } s\in S\}/ \sim,
\end{eqnarray*}
where two tuples $(\rho^{(i)},(P^{(i)}_s)_{s\in S})$ $(i=1,2)$ are
equivalent if there exists $P\in\mathrm{GL}(A^r)$ such that
$$(\rho^{(1)},(P^{(1)}_s)_{s\in S})=(P^{-1}\rho^{(2)}P,(P^{-1}P^{(2)}_s)_{s\in
S}).$$  Note that the equation $P^{(1)}_s= P^{-1}P^{(2)}_s$ implies
that $P\equiv I\mod \mathfrak m_A$ since we assume $S$ is nonempty
and $P^{(1)}_s\equiv P^{(2)}_s\mod \mathfrak m_A=P_{0,s}$. 
Using the Schlessinger criteria
\cite[Theorem 2.11]{S}, one can show the
functor $F^\lambda$ is pro-representable. 

\begin{lemma} \label{fiber}
Let $\kappa_\Lambda[\epsilon]$ be the ring of dual numbers. The
$\kappa_\Lambda$-vector space $F^\lambda(\kappa_\Lambda[\epsilon])$ is finite dimensional. 

(i) Suppose that 
$r$ is invertible in $\kappa_\Lambda$. 
We have  $$\mathrm{dim}\,F^\lambda (\kappa_\Lambda[\epsilon]) =\mathrm{dim}\, H_c^1(X-S,\mathcal
End({\mathcal F_0}))-2g,$$ where $\mathcal F_0$ is the lisse $\kappa_\Lambda$-sheaf on $X-S$ corresponding to 
the representation $\rho_0$.  

(ii) Suppose $\mathrm{End}_{\pi_1(X-S,\bar\eta)}(\kappa_\Lambda^r)(\cong \mathrm{End}(\mathcal F_0))$
consists of scalar multiplications, where $\kappa_\Lambda^r$ is considered as a
$\pi_1(X-S,\bar\eta)$-module through the representation $\rho_0$.  Suppose furthermore that 
$r$ is invertible in $\kappa_\Lambda$. 
Then the functor $F^\lambda$ is smooth.  
\end{lemma}

We will prove lemma \ref{fiber} in \S 2. 
Denote by $R(\rho_\Lambda)$ the universal
deformation ring for the functor $F^\lambda$. It is a complete noetherian local
$\Lambda$-algebra with residue field $\kappa_\Lambda$. We have a 
homomorphism
$$\rho_{\mathrm{univ}}:\pi_1(X-S,\bar\eta)\to\mathrm{GL}(r, R(\rho_\Lambda))$$
with the properties $$\rho_{\mathrm{univ}}\mod \mathfrak
m_{R(\rho_\Lambda)}=\rho_0,\quad \mathrm{det}(\rho_{\mathrm{univ}})=\lambda$$ and we have matrices $P_{\mathrm{univ},s}\in
\mathrm{GL}(r, R(\rho_\Lambda))$ with the properties
$$P_{\mathrm{univ},s}\mod \mathfrak
m_{R(\rho_\Lambda)}=P_{0,s},\quad P_{\mathrm{univ},s}^{-1}
\rho_{\mathrm{univ}}|_{\mathrm{Gal}(\bar\eta_s/\eta_s)} P_{\mathrm{univ},s}=\rho_\Lambda|_{\mathrm{Gal}
(\bar\eta_s/\eta_s)}$$ such that the homomorphism
$\pi_1(X-S,\bar\eta)\to\mathrm{GL}(r,R(\rho_\Lambda)/\mathfrak
m^i_{R(\rho_\Lambda)})$ induced by $\rho_{\mathrm{univ}}$ are
continuous for all positive integers $i$. It has the following universal property.

\begin{proposition} \label{baseextension} Let $A'$ be a local Artinian $\Lambda$-algebra so that its residue 
field $\kappa'=A'/\mathfrak m_{A'}$ is a finite extension of $\kappa_\Lambda=\Lambda/\mathfrak m$. Let $(\rho',(P'_s)_{s\in S})$
be a tuple such that $\rho':\pi_1(X-S,
\bar\eta)\to\mathrm{GL}({A'}^r)$ is a representation and $P'_s\in \mathrm{GL}(A'^r)$ with the properties
$$\rho'\mod\mathfrak m_{A'}=\rho_0, \quad P'_s \mod \mathfrak m_{A'}=P_{0,s},\quad \mathrm{det}(\rho')=\lambda,\quad
P_s^{\prime -1}\rho|_{\mathrm{Gal}(\bar\eta_s/\eta_s)} P'_s=\rho_\Lambda|_{\mathrm{Gal}
(\bar\eta_s/\eta_s)}\hbox{ for all } s\in S.$$ 
Then there exists a unique local
$\Lambda$-algebra homomorphism $R(\rho_\Lambda)\to A'$ which brings
$(\rho_{\mathrm{univ}},(P_{\mathrm{univ},s})_{s\in S}) $ to a tuple equivalent to 
$(\rho',(P'_s)_{s\in S})$, where the equivalence relation is defined as before.  
\end{proposition}

\begin{proof} When $\kappa'=\kappa$, this follows from the definition of the universal deformation ring. 
In general, let $A$ be the inverse image of $\kappa_\Lambda$ under the projection $A'\to A'/\mathfrak m_{A'}=\kappa'$. 
Then $A$ is an object in $\mathcal C_{\Lambda}$ and the tuple $(\rho, (P_s)_{s\in S})$ defines an element in $F^\lambda(A)$. 
We then apply the universal property of $R(\rho_A)$. 
\end{proof}

Let $E$ be a finite extension of $\mathbb Q_\ell$, let $\mathcal O$
be the integer ring of $E$, and let $\kappa$ be the residue field of $\mathcal O$.
Suppose $\mathcal F_E$ is a lisse $E$-sheaf on $X-S$ of rank
$r$. We say $\mathcal F_E$ is rigid if the corresponding $\overline{\mathbb Q}_\ell$-sheaf
$\mathcal F_E\otimes_E\overline{\mathbb Q}_\ell$ is rigid. 
Choose a torsion free lisse $\mathcal O$-sheaf $\mathcal
F_{\mathcal O}$ such that $\mathcal F_E\cong \mathcal F_{\mathcal
O}\otimes_{\mathcal O} E$. Let $\mathcal F_0= \mathcal F_{\mathcal
O}\otimes_{\mathcal O} \kappa$, let $\rho_{\mathcal O}:
\pi_1(X-S,\bar\eta)\to \mathrm{GL}(\mathcal O^r)$ be the
representation corresponding to the sheaf $\mathcal F_{\mathcal O}$,
let $\rho_E: \pi_1(X-S,\bar\eta)\to \mathrm{GL}(E^r)$ and
$\rho_0: \pi_1(X-S,\bar\eta)\to \mathrm{GL}(\kappa^r)$ be the
representations obtained from $\rho_{\mathcal O}$ by passing to the
fraction field and the residue field of $\mathcal O$, respectively, and let $\lambda=\mathrm{det}(\rho_{\mathcal O})$.
Take $\Lambda=\mathcal O$. 
Consider the universal deformation ring $R(\rho_{\mathcal O})$ of the 
functor $F^\lambda:\mathcal C_{\mathcal O}\to(\mathrm{Sets})$ for
the data $(\rho_0, (I)_{s\in S})$ where $P_{0,s}=I$ for all $s\in
S$.  As a local $\mathcal O$-algebra, $R(\rho_{\mathcal O})$ is isomorphic to a quotient of
$\mathcal O[[y_1,\ldots, y_n]]$ for some $n$. 

Let $D(0,1)$ be open unit disc considered as a rigid analytic space over $E$. 
Following Berthelot, we associate to $\mathcal O[[y_1,\ldots, y_n]]$
the rigid analytic space  $D(0,1)^n$. If $R$ is the quotient of 
$\mathcal O[[y_1,\ldots, y_n]]$ by an ideal generated by $g_1,\ldots,
g_k\in \mathcal O[[y_1,\ldots, y_n]]$, we associate to $R$ the closed
analytic subvariety $g_1=\ldots=g_k=0$ of $D(0,1)^n$. This rigid analytic space
can be thought as the generic fiber of
the formal scheme $\mathrm{Spf}\,R$ over $\mathrm{Spf}\,\mathcal O$. We refer the
reader to \cite[\S1]{Berk} and \cite[\S7]{dJ} for details of
Berthelot's construction.

Let $\mathfrak F^{\mathrm{rig}}$ be the rigid analytic space associated to  
the universal deformation ring $R(\rho_{\mathcal O})$ of the functor $F^\lambda:\mathcal C_{\mathcal O}
\to(\mathrm{Sets})$
for the data $(\rho_0,(I)_{s\in S})$.
By \cite[7.1.10]{dJ}, there is a one-to-one correspondence between 
the set of point in $\mathfrak F^{\mathrm{rig}}$ and the set of equivalent classes
of local homomorphisms 
$R(\rho_{\mathcal O})\to \mathcal O'$ of $\mathcal O$-algebras, where $\mathcal O'$ is the
integer ring of a finite extension $E'$ of $E$, and two such homomorphisms
$R(\rho_{\mathcal O})\to \mathcal O'$ and $R(\rho_{\mathcal O})\to \mathcal O''$ are equivalent 
if there exists a commutative diagram
$$\begin{array}{ccc}
R(\rho_{\mathcal O})&\to& \mathcal O'\\
\downarrow&&\downarrow\\
\mathcal O''&\to &\mathcal O'''
\end{array}$$ such that
$\mathcal O'''$ is the integer ring of a finite extension $E'''$ of $E$ containing both the fraction fields 
$E'$ and $E''$ of $\mathcal O'$ and $\mathcal O''$ respectively. 
Applying the universal property of $R(\rho_{\mathcal O})$ to the tuples
$(\rho_{\mathcal O}, (I)_{s\in S})\mod \mathfrak m_{\mathcal O}^i$
for all $i$,  we get a unique local $\mathcal O$-algebra
homomorphism $$\varphi_0: R(\rho_{\mathcal O})\to \mathcal O$$ which
brings the the universal tuple $(\rho_{\mathrm{univ}}, (P_{\mathrm{univ},s})_{s\in S})$ to a tuple 
equivalent to $(\rho_{\mathcal O},(I)_{s\in S})$. Replacing $(\rho_{\mathrm{univ}}, (P_{\mathrm{univ},s})_{s\in S})$
by an equivalent tuple if necessary, we may assume $\varphi_0$ brings the
universal representation
$\rho_{\mathrm{univ}}:\pi_1(X-S,\bar\eta)\to
\mathrm{GL}(r,R(\rho_{\mathcal O}))$ to $\rho_{\mathcal O}$, and
brings $P_{\mathrm{univ},s}$ to $I$ for all $s\in S$. The homomorphism $\varphi_0$ defines a point $t_0$ in
 $\mathfrak F^{\mathrm{rig}}$. Let $t$ be a point in
$\mathfrak F^{\mathrm{rig}}$  corresponding to a local homomorphism 
$$\varphi_t:R(\rho_{\mathcal O})\to \mathcal O'$$ of $\mathcal O$-algebras. 
Let $(\rho_t, (P_{t,s})_{s\in S})$ be the tuple obtained by pushing forward 
the universal tuple $(\rho_{\mathrm{univ}},(P_{\mathrm{univ},s})_{s\in S})$
through the homomorphism $\varphi_t$.  Note that
$\rho_t:\pi_1(X-S,\bar\eta)\to \mathrm{GL}(\mathcal O'^r)$ is 
a representation, $P_{t,s}\in \mathrm{GL}(\mathcal O'^r)$, and
\begin{eqnarray*}
&&\rho_t \mod\mathfrak m_{\mathcal O'}=\rho_0,\quad P_{t,s}\mod\mathfrak m_{\mathcal O'}=I,\quad \mathrm{det}(\rho_t)=\lambda, \quad P_{t,s}^{-1}\rho_t|_{\mathrm{Gal}(\bar\eta_s/\eta_s)} P_{t,s}=  \rho_{\mathcal O}|_{\mathrm{Gal}(\bar\eta_s/\eta_s)},\\
&& \varphi_{t_0}=\varphi_0, \quad \rho_{t_0}=\rho_{\mathcal O},\quad P_{t_0,s}=I,
\end{eqnarray*} 
Now suppose $\rho_E$ is
rigid.  Then the equations 
$\mathrm{det}(\rho_t)=\lambda$ and $P_{t,s}^{-1}\rho_t|_{\mathrm{Gal}(\bar\eta_s/\eta_s)} P_{t,s}=  \rho_{\mathcal O}|_{\mathrm{Gal}(\bar\eta_s/\eta_s)}$
$(s\in S)$ imply that $\rho_t$ is isomorphic to $\rho^{(j)}$ as $\overline{\mathbb Q}_\ell$-representations for some $j\in J$, where 
$\rho^{(j)}:\pi_1(X,\bar\eta)\to\mathrm{GL}(\overline{\mathbb Q}_\ell^r)$ is the representation corresponding to the $\overline{\mathbb Q}_\ell$-sheaf
$\mathcal F^{(j)}$ in the definition of the rigidity for $\mathcal F_E$. 
So there exists $P\in\mathrm{GL}(\overline{\mathbb Q}_\ell^r)$ such that $P^{-1}\rho_tP=\rho^{(j)}$. The following lemma of Deligne
shows that for those $t$ sufficiently close to $t_0$, we can get $\rho^{(j)}=\rho_E$, and 
we can choose $P$ so that $P\in \mathrm{GL}(\mathcal O^r_{\overline{\mathbb Q}_\ell})$ and 
$P\equiv I \mod \mathfrak m_{\mathcal O_{\overline{\mathbb Q}_\ell}}$, where $\mathcal O_{\overline{\mathbb Q}_\ell}$ 
is the integer ring of $\overline{\mathbb Q}_\ell$.  

\begin{lemma}[P. Deligne] \label{conj} Notation as above. Suppose 
$\mathcal F_E$ is absolutely irreducible and rigid. Then there
exists an admissible neighborhood $V$ of the point $t_0$ in $\mathfrak F^{\mathrm{rig}}$ such
that for any $t\in V$, there exists  $P\in \mathrm{GL}(\mathcal O_{\overline {\mathbb
Q}_\ell}^r)$ such that $P\equiv I \mod \mathfrak m_{{\mathcal O}_{\overline {\mathbb Q}_\ell}}$ and 
$P^{-1}\rho_tP=\rho_{\mathcal O}$.
\end{lemma}

We will give Deligne's proof of this lemma in \S 3. 

\medskip
Let $G:\mathcal C_{\mathcal O}\to(\mathrm{Sets})$ be the functor defined by
$$G(A)=\{(P_s)_{s\in S}| P_s \in
\mathrm{Aut}_{\mathrm{Gal}(\bar\eta_s/\eta_s)}(A^r),\; P_s\equiv
I\mod \mathfrak m_A\}/\sim,$$ where $A^r$ is provided with the
$\mathrm{Gal}(\bar\eta_s/\eta_s)$-action via $\rho_{\mathcal O}$,
and two tuples $(P_s^{(i)})_{s\in S}$ $(i=1,2)$ are equivalent if
there exists an invertible scalar $(r\times r)$-matrix $P= uI$ for
some unit $u$ in $A$ such that $P_s^{(1)}=P^{-1}P_s^{(2)}$ for all
$s\in S$. Using Schlessinger's criteria, one can verify that the functor
$G$ is pro-representable.  
Let $R(G)$ be the universal deformation ring for $G$. 
The rigid analytic space $\mathfrak G^{\mathrm{rig}}$ 
associated to $R(G)$ is a group object, and its points can be identified with the set 
$$\{(P_s)_{s\in S}| P_s \in
\mathrm{Aut}_{\mathrm{Gal}(\bar\eta_s/\eta_s)}({\mathcal O}_{\overline {\mathbb Q}_\ell}^r),\; P_s\equiv
I\mod \mathfrak m_{{\mathcal O}_{\overline {\mathbb Q}_\ell}}\}/\sim,$$ 
where ${\mathcal O}^r_{\overline{\mathbb Q}_\ell}$ is provided with the  
$\mathrm{Gal}(\bar\eta_s/\eta_s)$-action via $\rho_{\mathcal O}$, and 
two tuples $(P_s^{(i)})_{s\in S}$ $(i=1,2)$ are equivalent if there exists a unit $u$ in ${\mathcal O}_{\overline {\mathbb Q}_\ell}$
such that $P_s^{(1)}=u^{-1}P_s^{(2)}$ for all
$s\in S$.
Note that 
$$\mathrm{dim}\,\mathfrak G^{\mathrm{rig}}=\sum_{s\in S}\mathrm{dim}\,
\mathrm{End}_{\mathrm{Gal}(\bar\eta_s/\eta_s)}(E^r)-1=\sum_{s\in S}\mathrm{dim}\,
H^0(\eta_s, \mathcal End(\mathcal F_E|_{\eta_s}))-1,$$ where $E^r$ is provided with the
$\mathrm{Gal}(\bar\eta_s/\eta_s)$-action via $\rho_E$.

\medskip
Let $F^\lambda:\mathcal C_{\mathcal O}\to
(\mathrm{Sets})$ be the functor introduced before for the data
$(\rho_{\mathcal O},(I)_{s\in S})$. Note that for any tuple $(P_s)_{s\in S}$ defining an element in $G(A)$, 
$(\rho_{\mathcal O}, (P_s)_{s\in S})$ is a tuple defining an element in $F^\lambda(A)$. If two tuples
 $(P^{(i)}_s)_{s\in S}$ $(i=1,2)$ in $G(A)$ are equivalent, then the two tuples 
 $(\rho_{\mathcal O}, (P^{(i)}_s)_{s\in S})$ in $F^\lambda(A)$ are equivalent. 
 So we have a morphism of functors
$G\to F^\lambda$ defined by
$$G(A)\to F^\lambda(A),\quad (P_s)_{s\in S}\mapsto (\rho_{\mathcal O}, (P_s)_{s\in S}).$$ 

\begin{lemma}\label{surjection} Let $\mathfrak F^{\mathrm{rig}}$ (resp. $\mathfrak G^{\mathrm{rig}}$) be
the rigid analytic space associated to the universal deformation ring $R(\rho_{\mathcal O})$ (resp. $R(G)$) 
for the functor $F^\lambda$ (resp. $G$), 
and let  $f:\mathfrak G^{\mathrm{rig}}\to
\mathfrak F^{\mathrm{rig}}$ be the morphism on rigid analytic spaces induced by the
morphism of functors $G\to F^\lambda$. Suppose that $\mathcal F_E$ is absolutely 
irreducible and rigid. Let $V$ be the admissible neighborhood of $t_0$
in Lemma \ref{conj}.
Then $f:f^{-1}(V)\to V$ is surjective in the sense that every 
point in $V$ is the image of a point in 
$\mathfrak G^{\mathrm{rig}}$.
\end{lemma}

\begin{proof} Let $t$ be a point in $V$, and let
$\varphi_t:R(\rho_{\mathcal O})\to \mathcal O'$ be the corresponding local homomorphism of $\mathcal O$-algebras
(\cite[7.1.10]{dJ}),
where $\mathcal O'$ is the integer ring of a finite extension $E'$ of $E$. Let  $(\rho_t, (P_{t,s})_{s\in S})$ be the tuple 
obtained by pushing forward 
the universal tuple $(\rho_{\mathrm{univ}},(P_{\mathrm{univ},s})_{s\in S})$
through the homomorphism $\varphi_t$. By Lemma \ref{conj}, there exists $P\in\mathrm{GL} (\mathcal O^r_{\overline {\mathbb
Q}_\ell})$ such that $P\equiv I\mod \mathfrak m_{\mathcal O_{\overline {\mathbb
Q}_\ell}}$ and $P^{-1}\rho_t P=\rho_{\mathcal O}$. By enlarging $E'$, we may assume $P\in \mathrm{GL} (\mathcal O'^r)$. 
Then for each $i$, the tuple $(\rho_t,(P_{t,s})_{s\in S})\mod {\mathfrak
m}^i_{\mathcal O'}$ is equivalent to the tuple $(\rho_{\mathcal
O},(P^{-1}P_{t,s})_{s\in S})\mod {\mathfrak m}^i_{\mathcal O'}$. The family of
tuples $(P^{-1}P_{t,s})_{s\in S}\mod {\mathfrak m}^i_{\mathcal O'}$ defines a family of local $\mathcal O$-algebra homomorphisms
$R(G)\to \mathcal O'/{\mathfrak m}^i_{\mathcal O'}$. This family is compatible
and defines a local homomorphism $R(G )\to
\mathcal O'$ of $\mathcal O$-algebras. It corresponds to a point in $\mathfrak
G^{\mathrm{rig}}$ that is mapped by $f$ to the point $t$ of $\mathfrak
F^{\mathrm{rig}}$.
\end{proof}

The following result is due to Junyi Xie, whose proof is given in \S 4.

\begin{lemma}[J. Xie] \label{xie} Let $f:X\to Y$ be a morphism of rigid analytic spaces over 
a nonarchimedean field $K$ with a non-trivial valuation.
Suppose that $Y$  is separated and that 
$X$ can be covered by countably many  $K$-affinoid subdomains $X_n$ $(n=1,2,...)$. 
If $f$ is surjective on the underlying sets of points, then $\mathrm{dim}\,Y\leq 
\mathrm{dim}\,X$. 
\end{lemma}

The next result shows that the universal deformation ring $R(\rho_{\mathcal O})$ can be
used to recover the universal deformation rings for some $E'$-representations
of $\pi_1(X-S,\bar\eta)$. 

\begin{lemma} \label{genericfiber}  Let $E'$ be a finite extension of $E$, let $\mathcal O'$ be the 
integer ring of $E'$, let $\varphi_t: 
R(\rho_{\mathcal O})\to \mathcal O'$ be a local $\mathcal O$-algebra homomorphism,
let $(\rho_t, (P_{t,s})_{s\in S})$ be the tuple obtained by pushing forward 
the universal tuple $(\rho_{\mathrm{univ}},(P_{\mathrm{univ},s})_{s\in S})$
through the homomorphism $\varphi_t$, and let ${\mathfrak m}'_t$ (resp. ${\mathfrak m}_t$) be the kernel of the $E'$-algebra
(resp. $E$-algebra) homomorphism 
$$R(\rho_{\mathcal O})\otimes_{\mathcal O} E'\to E' \quad 
(\hbox{resp. }R(\rho_{\mathcal O})\otimes_{\mathcal O}E\to E')$$ induced by $\varphi_t$. 

(i) We have a canonical isomorphism 
$$(R(\rho_{\mathcal O})\otimes_{\mathcal O} E')_{{\mathfrak
m}'_t}^{\wedge}\cong R(\rho_t\otimes _{\mathcal O'}E'),$$ where $(R(\rho_{\mathcal O})\otimes_{\mathcal O}E')_{{\mathfrak
m}'_t}^{\wedge}$ is the completion of the local ring $(R(\rho_{\mathcal O})\otimes_{\mathcal O}E')_{{\mathfrak m}'_t}$,
and $R(\rho_t\otimes _{\mathcal O'}E')$ is the universal deformation ring of the functor
$F^\lambda:\mathcal C_{E'}\to (\mathrm{Sets})$ defined by 
\begin{eqnarray*}
F^\lambda(A)=\{(\rho,(P_s)_{s\in S})&|&\rho:\pi_1(X-S,
\bar\eta)\to\mathrm{GL}(A^r) \hbox{ is a representation},\; P_s\in \mathrm{GL}(A^r), \\
&&\rho\mod\mathfrak m_A=\rho_t, \quad P_s \mod \mathfrak m_A=P_{t,s},\\
&&\mathrm{det}(\rho)=\lambda,\quad P_s^{-1}\rho|_{\mathrm{Gal}(\bar\eta_s/\eta_s)} P_s=\rho_t|_{\mathrm{Gal}
(\bar\eta_s/\eta_s)}\hbox{ for all } s\in S\}/ \sim,
\end{eqnarray*}
for any local Artinian $E'$-algebra $A\in\mathrm{ob}\, \mathcal C_{E'}$ with residue field $E'$. 

(ii) Let $t$ be the point in $\mathfrak F
^{\mathrm{rig}}$ corresponding to $\varphi_t$. We have
$\widehat{\mathcal O}_{{\mathfrak F}^{\mathrm{rig}},t}\cong  (R(\rho_{\mathcal O})\otimes_{\mathcal O}E)_{\mathfrak
m_t}^{\wedge}.$
\end{lemma}

\begin{proof} Using Lemma \ref{baseextension}, we can reduce the proof of (i) to the case where $E'=E$. This 
case can be proved by modifying the argument in B.
Conrad's lecture note \cite[\S 7]{C}. 
(ii) follows from
\cite[Lemma 7.1.9]{dJ}. \end{proof}

We are now ready to prove Theorem \ref{maintheorem}.

\begin{proof}[Proof of Theorem \ref{maintheorem}] Suppose 
$\mathcal F_E$ is absolutely irreducible and rigid. Let's prove 
$\mathrm{dim}\, H^1(X,j_\ast \mathcal End(\mathcal F_E))= 2g$ and 
$H^1(X,j_\ast \mathcal End^{(0)}(\mathcal F_E))=0$.

Let $\varphi_0:R(\rho_{\mathcal O})\to \mathcal O$ 
be the local homomorphism of $\mathcal O$-algebras corresponding to 
the point $t_0$ in ${\mathfrak F}^{\mathrm{rig}}$, and let ${\mathfrak m}_0=\mathrm{ker}(\varphi_0\otimes{\mathrm{id}}_E)$. 
By Lemma \ref{genericfiber}, we have 
$$R(\rho_E)\cong (R(\rho_{\mathcal O})\otimes_{\mathcal O}E)_{\mathfrak m_0}^\wedge\cong
\widehat {\mathcal O}_{\mathfrak F^{\mathrm{rig}},t_0}.$$ We apply 
Lemma \ref{fiber} to the case $\Lambda=E$. Note that $\kappa_\Lambda=E$ has characteristic $0$ and 
$r$ is always invertible in $\kappa_\Lambda$. Then $R(\rho_E)$ is a formally smooth $E$-algebra,
and hence its dimension coincides with the dimension of its Zariski tangent space which is  
$\mathrm{dim}\, H^1_c(X-S, \mathcal End(\mathcal
F_E))-2g.$ So we
have $$\mathrm{dim}\, \widehat {\mathcal O}_{\mathfrak
F^{\mathrm{rig}},t_0}=\mathrm{dim}\, H^1_c(X-S,\mathcal End(\mathcal
F_E))-2g.$$ Let $V$ be the
admissible neighborhood of $t_0$ in Lemmas
\ref{conj}. By Lemmas \ref{surjection} and \ref{xie}, we have $\mathrm{dim}\, V\leq \mathrm{dim}\, f^{-1}(V)$. So
we have 
$$
\mathrm{dim}\, \widehat {\mathcal O}_{\mathfrak
F^{\mathrm{rig}},t_0}\leq \mathrm{dim}\, V\leq \mathrm{dim}\, f^{-1}(V)\leq \mathrm{dim}\,
{\mathfrak G}^{\mathrm{rig}}=
\sum_{s\in S}\mathrm{dim}\,
H^0(\eta_s, \mathcal End(\mathcal F_E|_{\eta_s}))-1.
$$ Comparing with the above expression for $\mathrm{dim}\, \widehat {\mathcal O}_{\mathfrak
F^{\mathrm{rig}},t_0}$, we get
$$\mathrm{dim}\, H^1_c(X-S, \mathcal End(\mathcal
F_E))-2g\leq 
\sum_{s\in S}\mathrm{dim}\,
H^0(\eta_s, \mathcal End(\mathcal F_E|_{\eta_s}))-1.\eqno(1.1)$$
Note that $j_!\mathcal End(\mathcal F_E)$ is a subsheaf of $j_*\mathcal End(\mathcal F_E)$,
the quotient sheaf is a sky-scrapper sheaf supported on $S$, and 
$$\Big(j_*\mathcal End(\mathcal F_E)/j_!\mathcal End(\mathcal F_E)\Big)_{\bar s}\cong 
H^0(\eta_s, \mathcal End(\mathcal F_E|_{\eta_s}))$$ for any $s\in S$.
We have
$H_c^0(X-S, \mathcal End(\mathcal F_E))=0$
since $X-S$ is an affine curve. Moreover, we have $H^0(X,j_* \mathcal End(\mathcal F_E))\cong
\mathrm{End}(\mathcal F_E)$. 
So we have a long exact sequence 
$$0\to \mathrm{End}(\mathcal F_E)\to\bigoplus_{s\in S} H^0(\eta_s, \mathcal End(\mathcal F_E|_{\eta_s}))
\to H^1_c(X-S, \mathcal End(\mathcal F_E))\to H^1(X,j_\ast\mathcal End(\mathcal F_E))\to 0.$$
It follows that 
$$\mathrm{dim}\,  H^1(X,j_\ast\mathcal End(\mathcal F_E))=\mathrm{dim}\,  H^1_c(X-S,
\mathcal End(\mathcal F_E))-\sum_{s\in S}\mathrm{dim}\,
H^0(\eta_s, \mathcal End(\mathcal F_E|_{\eta_s}))+\mathrm{dim}\,\mathrm{End}(\mathcal F_E).$$
So by the inequality (1.1), we have 
$$\mathrm{dim}\,  H^1(X,j_\ast\mathcal End(\mathcal F_E))\leq 2g+ \mathrm{dim}\,\mathrm{End}(\mathcal F_E)-1= 2g.$$
Here we use the fact that $\mathcal F_E$ is absolutely irreducible and hence
$\mathrm{dim}\,\mathrm {End}(\mathcal F_E)=1$ by Schur's lemma.  
We have
\begin{eqnarray*}
j_*\mathcal End(\mathcal F_E)\cong j_*\mathcal End^{(0)}(\mathcal F_E)\oplus E,\quad \mathrm{dim}\, H^1(X,E)=2g.
\end{eqnarray*}
So we have  $$\mathrm{dim}\, H^1(X,j_\ast \mathcal End^{(0)}(\mathcal F_E))\leq 0.$$
Hence  $H^1(X,j_\ast \mathcal End^{(0)}(\mathcal F_E))=0$ and $\mathrm{dim}\, H^1(X,j_\ast \mathcal End(\mathcal F_E))= 2g$.
\end{proof}

\begin{remark} One might hope that the naive functor $\mathcal C_{\mathcal O}\to (\mathrm{Sets})$ of deformations of $\mathcal F_{\mathcal O}$ with 
undeformed local monodromy and fixed determinant is pro-representable and the tangent space at a point in the rigid 
analytic space associated to the universal deformation ring of this functor can be identified with $H^1(X, j_*\mathcal End^{(0)}(\mathcal F_E))$.
This would greatly simplify the proof of the main theorem. Unfortunately the
Schlessinger's criteria do not hold for this naive functor. This forces us to work with the framed deformation. 
\end{remark}

\section{Proof of Lemma \ref{fiber}}

The content of this section is standard in the deformation theory of Galois representations. 
We include it for lack of reference. The following proposition proves the first part of Lemma \ref{fiber}.

\begin{proposition} \label{tangentspace}

$F^\lambda(\kappa_\Lambda[\epsilon])$ is finite dimensional.
If $r$ invertible in ${\kappa_\Lambda}$, then
\begin{eqnarray*}
\mathrm{dim}\, F^\lambda({\kappa_\Lambda}[\epsilon])&=& \mathrm{dim}\, H^1_c(X-S, \mathcal End^{(0)}(\mathcal F_0))+|S|-1\\
&=& \mathrm{dim}\, H^1_c(X-S,\mathcal End(\mathcal F_0))-2g.
\end{eqnarray*}
\end{proposition}

\begin{proof}
Let $\mathrm{Ad}(\rho_0)$ be the ${\kappa_\Lambda}$-vector space $\mathrm{End}(\kappa_\Lambda^r)$ of $r\times
r$ matrices with entries in ${\kappa_\Lambda}$ on which $\pi_1(X-S,\bar \eta)$
and $\mathrm{Gal}(\bar\eta_s/\eta_s)$ act by the composition of
$\rho_0$ with the adjoint representation of $\mathrm{GL}(\kappa_\Lambda^r)$, let 
$\mathrm{Ad}^{(0)}(\rho_0)$ be the subspace of $\mathrm{Ad}(\rho_0)$
consisting of matrices of trace $0$, and let $Z^1(\pi_1(X-S,\eta),\mathrm{Ad}^{(0)}(\rho_0))$ 
(resp. $B^1(\pi_1(X-S,\eta),\mathrm{Ad}^{(0)}(\rho_0))$) be
the group of 1-cocycles (resp. 1-coboundaries). 
Consider a tuple $(\rho,(P_s)_{s\in S})$ in $F^\lambda({\kappa_\Lambda}[\epsilon])$, and write
$$\rho(g)=\rho_0(g)+\epsilon M(g)\rho_0(g), \quad P_s=P_{0,s}+\epsilon Q_sP_{0,s}$$
for some $(r\times r)$-matrices $M(g)$ and $Q_s$ with entries in
${\kappa_\Lambda}$. The condition that $\rho$ is a representation with the property $\mathrm{det}(\rho)=\lambda$
is equivalent to saying that  the map $$\pi_1(X-S,\bar\eta)\to \mathrm{End}^{(0)}(\kappa_\Lambda^r),\quad g\mapsto M(g)$$ 
is a 1-cocycle for $\mathrm{Ad}^{(0)}(\rho_0)$, and the condition 
$P_s^{-1}\rho|_{\mathrm{Gal}(\bar\eta_s/\eta_s)} P_s=\rho_\Lambda|_{\mathrm{Gal}
(\bar\eta_s/\eta_s)}$ is equivalent to $$M|_{\mathrm{Gal}(\bar\eta_s/\eta_s)}+dQ_s=0.$$ 
Let $U$ be the kernel of the composite of the canonical homomorphisms
$$Z^1(\pi_1(X-S, \bar\eta), \mathrm{Ad}^{(0)}(\rho_0))\to H^1(\pi_1(X-S, \bar\eta), \mathrm{Ad}^{(0)}(\rho_0))\to \bigoplus_{s\in S} 
H^1(\mathrm{Gal}(\bar\eta_s/\eta_s),\mathrm{Ad}(\rho_0)).$$ 
Then $M$ lies in $U$. We have 
\begin{eqnarray*}
\mathrm{dim}\, U
&=&\mathrm{dim}\,\mathrm{ker}\Big( H^1(\pi_1(X-S, \bar\eta), \mathrm{Ad}^{(0)}(\rho_0))\to \bigoplus_{s\in S} 
H^1(\mathrm{Gal}(\bar\eta_s/\eta_s),\mathrm{Ad}(\rho_0))\Big)\\
&&\quad +\mathrm{dim}\,B^1(\pi_1(X-S, \bar\eta), \mathrm{Ad}^{(0)}(\rho_0))\\
&=& \mathrm{dim}\,\mathrm{ker}\Big(H^1(\pi_1(X-S, \bar\eta), \mathrm{Ad}^{(0)}(\rho_0))\to \bigoplus_{s\in S} 
H^1(\mathrm{Gal}(\bar\eta_s/\eta_s),\mathrm{Ad}(\rho_0))\Big)
\\&&\quad +\mathrm{dim}\,\mathrm{End}^{(0)}(\kappa_\Lambda^r)-\mathrm{dim}\, H^0(\pi_1(X-S, \bar\eta), \mathrm{Ad}^{(0)}(\rho_0)).
\end{eqnarray*}
Given $M\in U$, the set of matrices $Q_s$ such that $M|_{\mathrm{Gal}(\bar\eta_s/\eta_s)}+dQ_s=0$ 
form a space of dimension $\mathrm{dim}\, H^0(\mathrm{Gal}(\bar\eta_s/\eta_s), \mathrm{Ad}(\rho_0)).$
It follows that the set of tuples $(M, (Q_s)_{s\in S})$ with $M:\pi_1(X-S,\bar\eta)\to \mathrm{End}^{(0)}(\kappa_\Lambda^r)$ being a  1-cocycle 
and  $M|_{\mathrm{Gal}(\bar\eta_s/\eta_s)}+dQ_s=0$ form a vector space of dimension 
\begin{eqnarray*}
&&\mathrm{dim}\,\mathrm{ker}\Big(H^1(\pi_1(X-S, \bar\eta), \mathrm{Ad}^{(0)}(\rho_0))\to \bigoplus_{s\in S} 
H^1(\mathrm{Gal}(\bar\eta_s/\eta_s),\mathrm{Ad}(\rho_0))\Big)
\\&&\quad +\mathrm{dim}\,\mathrm{End}^{(0)}(\kappa_\Lambda^r)-\mathrm{dim}\, H^0(\pi_1(X-S, \bar\eta), \mathrm{Ad}^{(0)}(\rho_0))
+\sum_{s\in S} \mathrm{dim}\,  H^0(\mathrm{Gal}(\bar\eta_s/\eta_s),\mathrm{Ad}(\rho_0)).
\end{eqnarray*}
Given two tuples  $(\rho^{(i)},(P^{(i)}_s)_{s\in S})$ $(i=1,2)$ in $F^\lambda({\kappa_\Lambda}[\epsilon])$, 
write
$$\rho^{(i)}(g)=\rho_0(g)+\epsilon M^{(i)}(g)\rho_0(g), \quad P^{(i)}_s=P_{0,s}+\epsilon Q^{(i)}_sP_{0,s}.$$
These two tuples are equivalent if and only if there exists a matrix $P=I+\epsilon Q$ such that 
$$M^{(1)}-M^{(2)}=dQ,\quad Q_s^{(1)}=Q_s^{(2)}-Q.$$ So the set $F^\lambda({\kappa_\Lambda}[\epsilon])$
of equivalent classes of tuples $(\rho,(P_s)_{s\in S})$ form a vector space of dimension 
\begin{eqnarray*}
\mathrm{dim}\,
F^\lambda({\kappa_\Lambda}[\epsilon])&=&\mathrm{dim}\,\mathrm{ker}\Big(H^1(\pi_1(X-S, \bar\eta), \mathrm{Ad}^{(0)}(\rho_0))\to \bigoplus_{s\in S} 
H^1(\mathrm{Gal}(\bar\eta_s/\eta_s),\mathrm{Ad}(\rho_0))\Big)
\\&&\quad +\mathrm{dim}\,\mathrm{End}^{(0)}(\kappa_\Lambda^r)-\mathrm{dim}\, H^0(\pi_1(X-S, \bar\eta), \mathrm{Ad}^{(0)}(\rho_0))
+\sum_{s\in S} \mathrm{dim}\,  H^0(\mathrm{Gal}(\bar\eta_s/\eta_s),\mathrm{Ad}(\rho_0))\\
&&\quad -\mathrm{dim}\, \mathrm{End}(\kappa_\Lambda^r).
\end{eqnarray*}
In particular, $F^\lambda({\kappa_\Lambda}[\epsilon])$ is finite dimensional.
If $r$ is invertible in $\kappa_\Lambda$, then we have 
$\mathrm{Ad}(\rho_0)\cong \mathrm{Ad}^{(0)}(\rho_0)\oplus {\kappa_\Lambda}$, and the above expression can be written as  
\begin{eqnarray*}
\mathrm{dim}\, F^\lambda({\kappa_\Lambda}[\epsilon])
&=& \mathrm{dim}\,\mathrm{ker}\Big(H^1(\pi_1(X-S, \bar\eta), \mathrm{Ad}^{(0)}(\rho_0))\to \bigoplus_{s\in S} 
H^1(\mathrm{Gal}(\bar\eta_s/\eta_s),\mathrm{Ad}^{(0)}(\rho_0))\Big)\\
&&\quad -\mathrm{dim}\, H^0(\pi_1(X-S, \bar\eta), \mathrm{Ad}^{(0)}(\rho_0))+\sum_{s\in S} \mathrm{dim}\,  H^0(\mathrm{Gal}(\bar\eta_s/\eta_s),\mathrm{Ad}^{(0)}(\rho_0))
+|S|-1
\\
&=&\mathrm{dim}\,\mathrm{ker}\Big( H^1(X-S, \mathcal End^{(0)}(\mathcal F_0))\to \bigoplus_{s\in S} H^1(\eta_s,
\mathcal End^{(0)}(\mathcal F_0|_{\eta_s}))\Big)\\&&\quad -\mathrm{dim}\, H^0(X-S, \mathcal End^{(0)}(\mathcal F_0))
+\sum_{s\in S} \mathrm{dim}\,  H^0(\eta_s,\mathcal End^{(0)}(\mathcal F_0|_{\eta_s}))
+|S|-1.
\end{eqnarray*} Here for the second equality, we use the fact that 
$$H^1(\pi_1(X-S,\eta),\mathrm{Ad}^{(0)}(\rho_0))\cong
H^1(X-S,\mathcal End^{(0)}(\mathcal F_0))$$ by \cite[Lemma 1.6]{F}.
Let $\Delta$ be the mapping cone of the canonical morphism $j_!\mathcal End^{(0)}(\mathcal F_0)\to Rj_\ast \mathcal End^{(0)}(\mathcal F_0)$. 
Then we have $\mathcal H^i(\Delta)=0$ for $i\not=0,1$ and $\mathcal H^i(\Delta)$ are sky-scrapper sheaves supported on $S$
with 
$$(\mathcal H^i(\Delta))_{\bar s}\cong H^i(\eta_s, \mathcal End^{(0)}(\mathcal F_0|_{\eta_s})).$$ 
Taking the long exact sequence of cohomology groups for the distinguished triangle 
$$
j_! \mathcal End^{(0)}(\mathcal F_0)\to Rj_* \mathcal End^{(0)}(\mathcal F_0)\to \Delta\to \eqno(2.1)
$$
and taking into account of the fact that $H^0_c(X-S, \mathcal End^{(0)}(\mathcal F_0))=0$ (since $X-S$ is affine), we 
get a long exact sequence 
\begin{eqnarray*}
&& 0\to H^0(X-S, \mathcal End^{(0)}(\mathcal F_0))\to \bigoplus_{s\in S} H^0(\eta_s,  \mathcal End^{(0)}(\mathcal F_0|_{\eta_s}))
\to 
H^1_c(X-S, \mathcal End^{(0)}(\mathcal F_0))\\
&&\qquad \to \mathrm{ker}\Big(H^1(X-S, \mathcal End^{(0)}(\mathcal F_0))\to  \bigoplus_{s\in S} H^1(\eta_s,  \mathcal End^{(0)}(\mathcal F_0|_{\eta_s}))\Big)\to 0.
\end{eqnarray*}
It follows that $$\mathrm{dim}\,F^\lambda ({\kappa_\Lambda}[\epsilon])=\mathrm{dim}\, H_c^1(X-S, \mathcal End^{(0)}(\mathcal F_0))+|S|-1.$$
We have 
$
\mathrm{dim}\,H_c^1(X-S, {\kappa_\Lambda})
=2g-1+|S|.$
So we have
\begin{eqnarray*}
\mathrm{dim}\, F^\lambda({\kappa_\Lambda}[\epsilon])&=&
\mathrm{dim}\, H_c^1(X-S, \mathcal End^{(0)}(\mathcal F_0))+\mathrm{dim}\,H_c^1(X-S, {\kappa_\Lambda})-2g\\
&=& \mathrm{dim}\,H_c^1(X-S, \mathcal End(\mathcal F_0))-2g.
\end{eqnarray*}
\end{proof}

Let $A'\to A$ be an epimorphism in the category $\mathcal C_\Lambda$
such that its kernel $\mathfrak a$ has the property $\mathfrak
m_{A'}\mathfrak a=0$. We can regard $\mathfrak a$ as a vector space
over ${\kappa_\Lambda}\cong\ A'/\mathfrak m_{A'}$. Let
$\rho:\pi_1(X-S,\bar\eta)\to \mathrm{GL}(A^r)$ be a representation
such that $\rho\mod \mathfrak m_A=\rho_0$. By \cite[Lemma
2.1]{F}, we have
$H^i(\pi_1(X-S,\bar\eta),\mathrm{Ad}(\rho_0)\otimes_{\kappa_\Lambda}\mathfrak a)=0$ for all $i\geq 2$. Since the obstruction classes to lifting 
$\rho$ lies in $H^2(\pi_1(X-S, \bar\eta), \mathrm{Ad}(\rho_0)\otimes_{\kappa_\Lambda}\mathfrak a)=0$, 
the representation $\rho$ can always be
lifted to a representation $\rho':\pi_1(X-S,\bar\eta)\to
\mathrm{GL}(A'^r)$. Let $P_s\in \mathrm{GL}(A^r)$ $(s\in S)$ such that $P_s\mod \mathfrak m_A=P_{s,0}$,
and let 
$\rho'_s: \mathrm{Gal}(\bar\eta_s/\eta_s)\to \mathrm{GL}(A'^r)$ $(s\in S)$ be 
representations so that $$\rho'_s\mod \mathfrak a=P_s^{-1}\rho|_{\mathrm{Gal}(\bar\eta_s/\eta_s)}P_s.$$ 
Choose liftings $P_s'\in \mathrm{GL}(A'^r)$ for $P_s$. Then each
$P_s'\rho'_sP_s'^{-1}$ is a lifting of
$\rho|_{\mathrm{Gal}(\bar\eta_s/\eta_s)}$. Now
$\rho'|_{\mathrm{Gal}(\bar\eta_s/\eta_s)}$ is also a lifting of
$\rho|_{\mathrm{Gal}(\bar\eta_s/\eta_s)}$. The continuous map $\delta_s:
\mathrm{Gal}(\bar\eta_s/\eta_s)\to
\mathrm{End}(\kappa_\Lambda^r)\otimes_{\kappa_\Lambda}\mathfrak a$ defined by
$$\rho'(g)=P_s'\rho'_s(g)P_s'^{-1}+\delta_s(g) P_s'\rho'_s(g)P_s'^{-1}\quad (g\in
\mathrm{Gal}(\bar\eta_s/\eta_s))$$ is a 1-cocycle. Let $[\delta_s]$
be the cohomology class of $\delta_s$ in
$H^1(\mathrm{Gal}(\bar\eta_s/\eta_s),\mathrm{Ad}(\rho_0)\otimes_{\kappa_\Lambda}\mathfrak
a)$ and let $c$ be the image of $([\delta_s])_{s\in S}$ in the
cokernel of the canonical homomorphism
$$H^1(\pi_1(X-S,\bar\eta), \mathrm{Ad}(\rho_0)\otimes_{\kappa_\Lambda}\mathfrak a)
\to \bigoplus_{s\in S}H^1(\mathrm{Gal}(\bar\eta_s/\eta_s),
\mathrm{Ad}(\rho_0)\otimes_{\kappa_\Lambda}\mathfrak a).$$ Using the long exact sequence of cohomology groups 
associated 
to the distinguished triangle (2.1) with $\mathcal End^{(0)}(\mathcal F_0)$ replaced by
$\mathcal End(\mathcal F_0)$, we can show the above cokernel can be considered as a subspace of
$H^2(X,j_\ast\mathcal  End(\mathcal  F_0))\otimes_{\kappa_\Lambda}\mathfrak
a$. So we can also regard
$c$ as an element of in $H^2(X,j_\ast\mathcal End(\mathcal
F_0))\otimes_{\kappa_\Lambda}\mathfrak a$. We call $c$ the \emph{obstruction
class to lifting $(\rho, (P_s)_{s\in S})$ with the prescribed local
data} $(\rho'_s)_{s\in S}$. For simplicity, in the sequel we 
call $c$ the obstruction class to lifting $\rho$. It is straightforward to show that
$c$ is independent of the choices of $\rho'$
and of $P_s'$, and that $c$
vanishes if and only if $(\rho, (P_s)_{s\in S})$ can be lifted to a
tuple $(\rho'',(P_s'')_{s\in S})$ such that
$\rho'':\pi_1(X-S,\bar\eta)\to\mathrm{GL}(A'^r)$ is a representation
lifting $\rho$, $P_s''\in\mathrm{GL}(A'^r)$ lift $P_s$ and
$P_s^{\prime\prime -1}\rho''|_{\mathrm{Gal}(\bar\eta_s/\eta_s)}
P^{\prime\prime}_s=\rho'_s$ for all $s\in S$.
Note that we have
\begin{eqnarray*}
\mathrm{det}(\rho'(g))=\mathrm{det}(\rho'_s(g))+\mathrm{Tr}(\delta_s(g))\mathrm{det}(\rho'_s(g)) \quad (g\in\mathrm{Gal}(\bar\eta_s/\eta_s).)
\end{eqnarray*}
It follows that the obstruction class to lifting
$\mathrm{det}(\rho)$ is the image of the obstruction class to
lifting $\rho$ under the homomorphism
$$H^2(X,j_\ast\mathcal  End(\mathcal  F_0))\otimes_{\kappa_\Lambda}\mathfrak a\to
H^2(X,{\kappa_\Lambda})\otimes_{\kappa_\Lambda}\mathfrak a$$ induced by
$\mathrm{Tr}:\mathcal  End(\mathcal  F_0)\to {\kappa_\Lambda}.$

\begin{lemma} \label{obstruction2} Notation as above. Suppose all elements in
$\mathrm{End}_{\pi_1(X-S,\bar\eta)}(\kappa_\Lambda^r)$ are scalar
multiplications and suppose $\mathrm{det}(\rho)$ can be lifted to a
representation $\lambda':\pi_1(X-S,\bar\eta)\to\mathrm{GL}(A')$
with the property
$\lambda'|_{\mathrm{Gal}(\bar\eta_s/\eta_s)}=\mathrm{det}\,\rho'_s$.
Then the tuple $(\rho, (P_s)_{s\in S})$ can be lifted to a tuple $(\rho',(P_s')_{s\in S})$
such that $\rho':\pi_1(X-S,\bar\eta)\to
\mathrm{GL}(A'^r)$ is a representation and 
$P_s^{\prime -1}\rho'|_{\mathrm{Gal}(\bar\eta_s/\eta_s)}
P_s^\prime=\rho'_s$ for all $s\in S$.
\end{lemma}

\begin{proof} If all elements in
$\mathrm{End}_{\pi_1(X-S,\bar\eta)}(\kappa_\Lambda^r)$ are scalar
multiplications, then the morphism
$${\kappa_\Lambda}\to\mathcal  End(\mathcal  F_0),\quad a\mapsto a\mathrm{Id}$$ induces an
isomorphism
$$H^0(X,{\kappa_\Lambda})\cong H^0(X,j_\ast\mathcal  End(\mathcal  F_0)).$$
By the Poincar\'e duality
(\cite[1.3 and 2.2]{Ddualite}), the morphism $\mathrm{Tr}:\mathcal  End(\mathcal  F_0)\to {\kappa_\Lambda}$
induces an isomorphism
$$H^2(X,j_\ast\mathcal  End(\mathcal  F_0))\stackrel\cong\to H^2(X,{\kappa_\Lambda}).$$
This last isomorphism maps the obstruction class
to lifting $\rho$ to the obstruction class to lifting
$\mathrm{det}(\rho)$. By our assumption, there is no obstruction to
lifting $\mathrm{det}(\rho)$. It follows that there is no
obstruction to lifting $(\rho,(P_s)_{s\in S})$ with the prescribed local data $(\rho'_s)$.
\end{proof}

\begin{proof}[Proof of Lemma \ref{fiber} (ii)] Suppose $(\rho, (P_s)_{s\in S})$ defines an element of
$F^\lambda(A)$. There is no obstruction to lifting $\mathrm{det}(\rho)=\lambda$ 
since $\lambda$ takes its value in the ground coefficient ring $\Lambda$.
By Lemma \ref{obstruction2}, we can
lift $(\rho,(P_s)_{s\in S})$ to a tuple $(\rho',(P'_s)_{s\in S})$ such that
$P_s'^{-1} \rho'|_{\mathrm{Gal}(\bar\eta_s/\eta_s)}P'_s=\rho_\Lambda|_{\mathrm{Gal}(\bar\eta_s/\eta_s)}$. We then
have
$$\mathrm{det}(\rho')|_{\mathrm{Gal}(\bar\eta_s/\eta_s)}=\lambda|_{\mathrm{Gal}(\bar\eta_s/\eta_s)}.$$ In particular,
$\lambda^{-1}\mathrm{det}(\rho')$ is unramified at each $s\in S$. It is
also unramified on $X-S$. So $\lambda^{-1}\mathrm{det}(\rho')$ defines a character $\chi:\pi_1(X,\bar\eta)\to
A^{\prime *}.$
Note that we have $$\lambda^{-1}\mathrm{det}(\rho')\mod \mathfrak a=\lambda^{-1}\mathrm{det}(\rho)=1.$$
So the image of $\chi$ lies in the subgroup
$1+\mathfrak m_{A'}$ of $A^{\prime\ast}$. This subgroup has a filtration
$$1+\mathfrak m_{A'}\supset 1+\mathfrak m_{A'}^2\supset \cdots.$$ For each $i$,
we have an isomorphism of groups
$\mathfrak m_{A'}^i/\mathfrak m_{A'}^{i+1}\cong
(1+\mathfrak m_{A'}^i)/(1+\mathfrak m_{A'}^{i+1}),$ and $\mathfrak
m_{A'}^i/\mathfrak m_{A'}^{i+1}$ is the underlying abelian group of a
finite dimensional vector space over ${\kappa_\Lambda}$. Any
profinite subgroup of a finite dimensional ${\kappa_\Lambda}$-vector space must be
a pro-$\ell$-group. It follows that the representation
$\chi:\pi_1(X,\bar\eta)\to A^{\prime\ast}$ must factor through the maximal
abelian pro-$\ell$-quotient $\Gamma$ of $\pi_1(X,\bar\eta)$. By \cite[X
3.10]{SGA1}, we have 
$\Gamma\cong\mathbb Z_\ell^{2g}$. Since $r$ is invertible in $\kappa_\Lambda$, 
any element in $1+\mathfrak m_{A'}$ has an 
$r$-th root lying in $1+\mathfrak m_{A'}$. So any character 
$\mathbb Z_\ell^{2g}\to 1+\mathfrak m_{A'}$ is the $r$-th power of another such character.  We 
can thus find a character $\chi': \pi_1(X,\bar\eta)\to 1+\mathfrak m_{A'}$ such that $\chi'^r=\chi$. Then 
$$\mathrm{det}(\rho'\chi'^{-1})=\mathrm{det}(\rho')\chi'^{-r}=\mathrm{det}(\rho')\chi^{-1}=\lambda.$$ 
So $(\rho'\chi'^{-1}, (P'_s)_{s\in S})$ defines an element of
$F^\lambda(A')$ lifting $(\rho, (P_s)_{s\in S})$. Therefore $F^\lambda$ is smooth.
\end{proof}

\section{Proof of Deligne's Lemma \ref{conj}}

\begin{lemma}\label{little} Let $E$ be a finite extension of $\mathbb Q_\ell$, let $\pi$
be a uniformizer for the integer ring $\mathcal O$ of $E$, let $\mathcal F_{\mathcal O}$ 
be a torsion free lisse $\mathcal O$-sheaf
on $X-S$, and let $\mathcal F_E=\mathcal F_{\mathcal O}\otimes_{\mathcal O}E$. 
Suppose ${\mathcal F_E}$ is absolutely irreducible. 

(i) There exists a natural  number $N$ such that for any homomorphism 
$\phi:\mathcal F_{\mathcal O}/\pi^N\mathcal F_{\mathcal O}\to \mathcal F_{\mathcal O}/\pi^N\mathcal F_{\mathcal O},$
there exists a scalar $a\in\mathcal O$ such that modulo $\pi$, $\phi$ coincides with the
scalar multiplication by $a$ on $\mathcal F_{\mathcal O}/\pi\mathcal F_{\mathcal O}$.

(ii) Let $\mathcal E_{\mathcal O}$ be a torsion free lisse $\mathcal O$-sheave on $X-S$ of the same rank as $\mathcal F_{\mathcal O}$
such that $\mathcal E_{\mathcal O}\otimes_{\mathcal O}\overline{\mathbb Q}_\ell$ is not isomorphic to 
$\mathcal F_{\mathcal O}\otimes_{\mathcal O}\overline{\mathbb Q}_\ell.$ 
There exists a natural  number $N$ such that any homomorphism 
$\mathcal E_{\mathcal O}
/\pi^N\mathcal E_{\mathcal O}\to \mathcal F_{\mathcal O}/\pi^N\mathcal F_{\mathcal O}
$ vanishes modulo $\pi$. 
\end{lemma}

\begin{proof} For any natural number $n$, let $S_n$ be the subset of $\mathrm{End}(\mathcal F_{\mathcal O}/\pi^n\mathcal F_{\mathcal O})$ 
consisting of those endomorphisms $\phi_n$ such that $\phi_n\mod \pi$ are not scalar multiplications. Note that each $S_n$ is a finite
set, and we have a map $S_n\to S_{n-1}$ sending each $\phi_n$ in $S_n$ to $\phi_n \mod \pi^{n-1}$. 
If each $S_n$ is nonempty, then $\varprojlim_n S_n$ is nonempty. Choose an element $(\phi_n)\in\varprojlim_n S_n$. The family of 
compatible endomorphisms $\phi_n:\mathcal F_{\mathcal O}/\pi^n\mathcal F_{\mathcal O}\to \mathcal F_{\mathcal O}/\pi^n\mathcal F_{\mathcal O}$
induces an endomorphism $\phi:\mathcal F_{\mathcal O}\to\mathcal F_{\mathcal O}$ by passing to inverse limit. By Schur's lemma, 
$\mathrm{End}({\mathcal F_E})$ consists of scalar multiplications, so $\phi$ is necessarily a scalar multiplication by some element 
$a\in\mathcal O$. But then each $\phi_n$ is a scalar multiplication. This contradicts to the fact that $\phi_n\in S_n$. So there exists
an integer $N$ such that $S_N$ is empty. This proves (i).

The proof of (ii) is similar. We take $S_n$ to be the subset of $\mathrm{Hom}(\mathcal E_{\mathcal O}
/\pi^n \mathcal E_{\mathcal O}, \mathcal F_{\mathcal O}/\pi^n\mathcal F_{\mathcal O})
$ consisting of homomorphisms $\phi_n$ so that $\phi_n\mod\pi$ are not zero. By Schur's lemma, 
we have $\mathrm{Hom}(\mathcal E_E,\mathcal F_E)=0$. We use this fact to get contradiction as above if 
$S_n$ are nonempty for all $n$. 
\end{proof}

\begin{lemma}\label{big} Keep the assumption in Lemma \ref{little}. 
For any finite extension $E'$ of $E$, let $\mathcal O'$ be the integer ring of $E'$, 
let $\mathcal F_{\mathcal O'}=\mathcal F_{\mathcal O}\otimes_{\mathcal O}\mathcal O'$
and $\mathcal E_{\mathcal O'}=\mathcal E_{\mathcal O}\otimes_{\mathcal O}\mathcal O'$.

(i) There exists an integer 
$N$ depending only on $\mathcal F_{\mathcal O}$ such that for any endomorphism $\phi':
\mathcal F_{\mathcal O'}/\pi^N\mathcal F_{\mathcal O'}\to 
\mathcal F_{\mathcal O'}/\pi^N\mathcal F_{\mathcal O'}$, there exists a scalar $a\in\mathcal O'$ such  
that modulo $\pi$, $\phi'$ coincides with the
scalar multiplication by $a$ on $\mathcal F_{\mathcal O'}/\pi\mathcal F_{\mathcal O'}$.

(ii) There exists an integer 
$N$ depending only on $\mathcal F_{\mathcal O}$ and $\mathcal E_{\mathcal O}$ such that
any homomorphism $\mathcal E_{\mathcal O'}
/\pi^N\mathcal E_{\mathcal O'}\to \mathcal F_{\mathcal O'}/\pi^N\mathcal F_{\mathcal O'}
$ vanishes modulo $\pi$. 
\end{lemma}

\begin{proof} We prove (i). The proof of (ii) is similar. 
Let's prove the integer $N$ in Lemma \ref{little} has the required property. 
Let $\rho_{\mathcal O}:\pi_1(X-S,\bar\eta)\to\mathrm{GL}(\mathcal O^r)$ be the representation 
corresponding to $\mathcal F_{\mathcal O}$. An endomorphism $\phi'$ of 
$\mathcal F_{\mathcal O'}/\pi^N\mathcal F_{\mathcal O'}$ corresponds to an endomorphism
of the $\pi_1(X-S,\bar\eta)$-module $(\mathcal O'/\pi^N\mathcal O')^r$ which we still denote by 
$\phi'$, where $\pi_1(X-S,\bar\eta)$ acts on $(\mathcal O'/\pi^N\mathcal O')^r$ via the representation 
$\rho_{\mathcal O}$. Note that $\mathcal O'$ is a free $\mathcal O$-module of finite rank. So $\mathcal O'/\pi^N\mathcal O'$
is a free $\mathcal O/\pi^N\mathcal O$-module of finite rank. Choose a basis $\{e_1,\ldots, e_m\}$ of 
$\mathcal O'/\pi^N\mathcal O'$ over $\mathcal O/\pi^N\mathcal O$.
Let $\Phi'$ be the matrix for the endmorphism $\phi'$ on $(\mathcal O'/\pi^N\mathcal O')^r$ 
with respect to the standard basis of $(\mathcal O'/\pi^N\mathcal O')^r$. The entries of 
$\Phi'$ lie in $\mathcal O'/\pi^N\mathcal O'$. We can 
write 
$$\Phi'=e_1\Phi_1+\cdots+ e_m\Phi_m$$ 
for some uniquely determined matrices $\Phi_i$ $(i=1,\ldots, m)$ with entries lying in $\mathcal O/\pi^N\mathcal O$.
For any $g\in\pi_1(X-S,\bar\eta)$, we have $\Phi' \rho_{\mathcal O}(g)=\rho_{\mathcal O}(g)\Phi'$, that is, 
$$e_1\Phi_1\rho_{\mathcal O}(g)+\cdots+e_m\Phi_m\rho_{\mathcal O}(g)=e_1\rho_{\mathcal O}(g)\Phi_1+\cdots
+e_m\rho_{\mathcal O}(g)\Phi_m.$$
This implies that $\Phi_i\rho_{\mathcal O}(g)=\rho_{\mathcal O}(g)\Phi_i$ for each $i$. 
So $\Phi_i$ corresponds to an endomorphism of the $\pi_1(X-S,\bar\eta)$-module $(\mathcal O/\pi^N\mathcal O)^r$
defined by the representation $\rho_{\mathcal O}$. Hence it defines an endomorphism $\phi_i$ on $\mathcal F_{\mathcal O}
/\pi^N\mathcal F_{\mathcal O}$. 
By Lemma \ref{little}, there exists a scalar $a_i\in \mathcal O$ such that modulo $\pi$, $\phi_i$ coincides with
the scalar multiplication by $a_i$.  Then modulo $\pi$, $\phi'$ coincides with the scalar multiplication by 
$a=e_1a_1+\cdots+ e_ma_m$.  
\end{proof}

We are now ready to prove Lemma \ref{conj}.

\begin{proof}[Proof of Lemma \ref{conj}] Suppose $\mathcal F_E$ is rigid. Let $\mathcal F^{(j)}$ $(j\in J)$ be a finite family of lisse
$\overline{\mathbb Q}_\ell$-sheaves on $X-S$ such that for any lisse $\overline{\mathbb Q}_\ell$-sheaf $\mathcal G$ on 
$X-S$ with the property  $$\mathrm{det}(\mathcal F_E\otimes_E\overline {\mathbb Q}_\ell)\cong\mathrm{det}(\mathcal G),\quad
(\mathcal F_E\otimes_E\overline {\mathbb Q}_\ell)|_{\eta_s}\cong \mathcal G|_{\eta_s}\hbox{ for all }
s\in S,$$ we have $\mathcal G\cong \mathcal F^{(j)}$ for some $j\in J$. By enlarging $E$, we may assume there exist lisse $E$-sheaves 
$\mathcal F^{(j)}_E$ on $X-S$ such that $\mathcal F^{(j)}_E\otimes_E \overline{\mathbb Q}_\ell\cong \mathcal F^{(j)}$.
Let $\mathcal F^{(j)}_{\mathcal O}$ be torsion free lisse $\mathcal O$-sheaves on $X-S$ such that  
$\mathcal F^{(j)}_{\mathcal O}\otimes_{\mathcal O}E\cong\mathcal F^{(j)}_E$. Moreover, we 
assume $\mathcal F^{(j)}$ $(j\in J)$ are pairwise non-isomorphic
and all have rank $r$. In particular, there exists one and only one $j_0\in J$ such that $\mathcal F^{(j_0)}\cong \mathcal F_E\otimes_E \overline{\mathbb Q}_\ell$. 
We take $\mathcal F^{(j_0)}_{\mathcal O}=\mathcal F_{\mathcal O}$. Since $J$ is finite, we can take a sufficiently large $N$
so that Lemma \ref{big} (i) holds, and \ref{big} (ii) holds for $\mathcal E_{\mathcal O}
=\mathcal F^{(j)}_{\mathcal O}$ for each $j\not= j_0$. 

Recall that the local $\mathcal O$-algebra homomorphism $\varphi_0: R(\rho_{\mathcal O})
\to \mathcal O$ corresponding to the point $t_0$ in $\mathfrak F^{\mathrm{rig}}$ brings the universal tuple 
$(\rho_{\mathrm{univ}}, (P_{\mathrm{univ},s})_{s\in S})$ to $(\rho_{\mathcal O},(I)_{s\in S})$. Let $a_1,\ldots,a_n$ be a family of generators 
for the ideal $\mathrm{ker}(\varphi_0)$. One can prove that the $\mathcal O$-algebra homomorphism
$$\mathcal O[[T_1,\ldots, T_n]]\to R(\rho_{\mathcal O}),\quad T_i\mapsto a_i$$
is surjective using \cite[10.23]{AM}. Let $\mathfrak a$ be its kernel.  Then $\mathfrak F^{\mathrm{rig}}$ is defined  be the zero set of the 
ideal $\mathfrak a$ in the open unit polydisc $D(0,1)^n$. Consider the admissible subdomain $V$ of $\mathfrak F^{\mathrm{rig}}$ defined by 
$$|t_i|\leq |\pi^N|\quad (i=1,\ldots, n).$$ Let $t=(t_1,\ldots, t_n)$ be a point in $V$, 
and let $\varphi_t:R(\rho_{\mathcal O})\to \mathcal O'$ be the corresponding 
local homomorphism of $\mathcal O$-algebras, where $\mathcal O'$ is the integer ring of some finite extension $E'$ of $E$. 
Note that via the isomorphism $\mathcal O[[T_1,\ldots, T_n]]/\mathfrak a\cong R(\rho_{\mathcal O}),$ $\varphi_t$ is identified with the homomorphism
$$\mathcal O[[T_1,\ldots, T_n]]/\mathfrak a\to \mathcal O',\quad f(T_1,\ldots, T_n)\mapsto f(t_1,\ldots, t_n),$$ and $\varphi_0:R(\rho_{\mathcal O})\to\mathcal O$ 
is identified with the homomorphism $$\mathcal O[[T_1,\ldots, T_n]]/\mathfrak a\to \mathcal O,\quad f(T_1,\ldots, T_n)\mapsto f(0,\ldots, 0).$$ As 
$|t_i|\leq |\pi^N|$, we have 
 $$\varphi_t\equiv \varphi_0\mod \pi^N\mathcal O'.$$ 
Let $(\rho_t,(P_{t,s})_{t\in S})$ be the tuple obtained by pushing-forward
the universal tuple 
$(\rho_{\mathrm{univ}},(P_{\mathrm{univ},s})_{s\in S})$ through $\varphi_t$. Then we have 
$$\rho_t\equiv \rho_{\mathcal O} \mod \pi^N \mathcal O',\quad \mathrm{det}(\rho_t)=\lambda,\quad
P_{t,s}^{-1} \rho_t|_{\mathrm{Gal}(\bar\eta_s/\eta_s)} P_{t,s}=\rho_{\mathcal O}|_{\mathrm{Gal}(\bar\eta_s/\eta_s)}.$$ 
Since $\rho_E$ is rigid, the last two equalities above imply that there exists $j_1\in J$ such that
$\rho_t\cong \rho^{(j_1)}_E$ as $\overline{\mathbb Q}_\ell$-representations, where we denote by
$\rho_E^{(j)}:\pi_1(X-S,\bar\eta)\to \mathrm{GL}(E^r)$ the representation corresponding to the sheaf $\mathcal F^{(j)}_E$. 
Choose $E'$ sufficiently large so that 
$\rho_t\cong  \rho^{(j_1)}_E$ as $E'$-representations. 
Let $\mathcal F_t$ be the torsion free  lisse $\mathcal O'$-sheaf on $X-S$ defined by the 
representation $\rho_t:\pi_1(X-S,\bar\eta)\to \mathrm{GL}(\mathcal O'^r)$. Then we have an isomorphism 
$\mathcal F^{(j_1)}_{\mathcal O}\otimes_{\mathcal O}E' \cong \mathcal F_t
\otimes_{\mathcal O'}E'.$
Let $
\mathcal F^{(j_1)}_{\mathcal O'}=\mathcal F^{(j_1)}_{\mathcal O}\otimes_{\mathcal O}\mathcal O'.$
We can choose a homomorphism of $\mathcal O'$-sheaves
$\alpha:
\mathcal F^{(j_1)}_{\mathcal O'}\to \mathcal F_t
$
so that after tensoring with $E'$, it 
induces an isomorphism $\mathcal F^{(j_1)}_{\mathcal O}\otimes_{\mathcal O} E' \cong \mathcal F_t\otimes_{\mathcal O'}E'$, 
and that it is nonzero modulo $\mathfrak m_{\mathcal O'}$. 
The identity  $\rho_t\equiv \rho_{\mathcal O} \mod \pi^N \mathcal O'$
defines an isomorphism 
$\beta:\mathcal F_t/\pi^N\mathcal F_t\stackrel\cong\to \mathcal F_{\mathcal O'}/\pi^N\mathcal F_{\mathcal O'},$
where $\mathcal F_{\mathcal O'}=\mathcal F_{\mathcal O}\otimes_{\mathcal O}\mathcal O'$.
Consider the composite 
$$\mathcal F^{(j_1)}_{\mathcal O'}/\pi^N\mathcal F^{(j_1)}_{\mathcal O'}
 \stackrel{\bar\alpha}\to 
\mathcal F_t/\pi^N\mathcal F_t\stackrel\beta\to \mathcal F_{\mathcal O'}/\pi^N\mathcal F_{\mathcal O'},$$
where the first homomorphism $\bar\alpha$ is induced by $\alpha$. By the choice of $\alpha$ and $\beta$, modulo ${\mathfrak m}_{\mathcal O'}$, 
$\beta\bar\alpha$ is nonzero. By Lemma \ref{big} (ii), we must have $\mathcal F^{(j_1)}\cong \mathcal F_E\otimes_E \overline{\mathbb Q}_\ell$, 
and hence $j_1=j_0$. From now on, 
we replace $\mathcal F^{(j_1)}_{\mathcal O}$ by $\mathcal F_{\mathcal O}$. 
By Lemma \ref{big} (i), modulo $\pi$, $\beta\bar\alpha$ 
coincides with a scalar multiplication on $\mathcal F_{\mathcal O'}/\pi\mathcal F_{\mathcal O'}$ 
by some $a\in\mathcal O'$. By the choice of $\alpha$ and $\beta$, modulo ${\mathfrak m}_{\mathcal O'}$, 
$\beta\bar\alpha$ is a nonzero
endomorphism of $\mathcal F_{\mathcal O'}/{\mathfrak m}_{\mathcal O'}\mathcal F_{\mathcal O'}$.  So modulo ${\mathfrak m}_{\mathcal O'}$,
$\beta\bar\alpha$ is the scalar multiplication on $\mathcal F_{\mathcal O'}/{\mathfrak m}_{\mathcal O'}\mathcal F_{\mathcal O'}$
by a nonzero element in the field $\mathcal O'/{\mathfrak m}_{\mathcal O'}$. 
In particular, 
$a$ is a unit in $\mathcal O'$, and modulo ${\mathfrak m}_{\mathcal O'}$,
$\beta\bar\alpha$ is an isomorphism. This implies that $\beta\bar\alpha$ itself is an isomorphism. Since $\beta$ is an isomorphism,
$\bar\alpha$ must be an isomorphism. This implies that $\alpha$ is an isomorphism.  The 
isomorphism $a^{-1}\alpha:\mathcal F_{\mathcal O'}\to \mathcal F_t$ defines an isomorphism of representations
$P:\mathcal O'^r\to \mathcal O'^r$ from $\rho_{\mathcal O}$ to $\rho_t$. We have 
$P\in \mathrm{GL}(\mathcal O'^r)$, 
$\rho_t P=P \rho_{\mathcal O}$ and $P\equiv I\mod\mathfrak m_{\mathcal O'}$. This proves our assertion. 
\end{proof}

\section{Proof of Xie's Lemma \ref{xie}}

Making a base change to the completion of the algebraic
closure of $K$, we may assume $K$ is a complete algebraically closed field. We may reduced to the case where 
$Y=\mathrm{Sp}\, A$ for a 
strictly $K$-affinoid algebra $A$ (in the language of Berkovich \cite{B2}). Then by definition, $\mathrm{dim}\,Y$ is 
the Krull dimension of $A$. By the Noether normalization theorem (\cite [6.1.2/2]{BGR}), there exists a finite monomorphism 
$T_d\to A$ for some Tate algebra $T_d=k\langle t_1,\ldots, t_d\rangle$ with $d=\mathrm{dim}\,Y$. Note that the induced morphism 
$Y\to \mathrm{Sp}\, T_d$ is surjective on the underlying set of points. So we can reduce to the case where $A=T_d$
and $Y=E(0,1)^d$ is the closed unit polydisc of dimension $d$. 

Let $\mathcal X$ (resp. $\mathcal X_n$, resp. $\mathcal Y$) 
be the Berkovich space associated to $X$(resp. $X_n$, resp.  $Y$). Denote the 
morphism $\mathcal X\to \mathcal Y$ corresponding to $f:X\to Y$ also by $f$. 
For any real tuple $\underline r=(r_1, \ldots, r_d)$ with $0<r_i\leq 1$ and $r_i\in |K^\ast|$, and 
any rigid point $\underline a =(a_1,\ldots,a_d)$ in $E(0,1)^d$, where $a_i\in K$ and $|a_i|\leq 1$, consider the polydisc
$$E(\underline a,\underline r)=\{(x_1,\ldots, x_d)\in K^d: |x_i-a_i|\leq r_i\}.$$  We have $E(\underline a,\underline r)\subset
E(0,1)^d$. 
We define the associated Gauss
norm $|\cdot|_{E(\underline a,\underline r)}$ on $T_d$ by 
$$|f|_{E(\underline a,\underline r)}
=\mathrm{max} \{ |f(x)| :x\in E(\underline a,\underline r)\}$$ for any $f\in T_d$.  
The Gauss norms $|\cdot|_{E(\underline a,\underline r)}$ are points in $\mathcal Y$. 
Let $\mathcal S$ be the subset of $\mathcal Y$ consisting of all Gauss norms 
associated to all polydiscs $E(\underline a,\underline r)$. Note that $\mathcal S$ is dense in $\mathcal Y$. Indeed, as the radius
$\underline r=(r_1,\ldots, r_d)$ approaches to $0$, the Gauss norm $|\cdot|_{E(\underline a,\underline r)}$ approaches to the rigid point $\underline a$, and 
it is known that the set of all rigid points is dense in $\mathcal Y$ (\cite[2.1.15]{B2}). Moreover, for any $y\in\mathcal S$, 
one can use \cite[5.1.2/2]{BGR} to show that $s(\mathcal H(y)/K)=d$,
where $\mathcal H(y)$ is the field defined in \cite[1.2.2 (i)]{B2}, and $s(\mathcal H(y)/K)=\mathrm{tr.deg}(\widetilde{\mathcal H(y)}/\widetilde K)$ is defined 
in \cite[9.1]{B2}.
We claim that $f(\mathcal X)\cap \mathcal S$ in nonempty. Otherwise, 
$f(\mathcal X_n)$ is disjoint from $\mathcal S$ for each $n$, that is, $\mathcal S\subset \mathcal Y-
f(\mathcal X_n)$. Hence $\mathcal Y-f(\mathcal X_n)$
is dense in $\mathcal Y$. Since $\mathcal X_n$ is affiniod, 
it is compact (\cite[1.2.1]{B2}). So $f(\mathcal X_n)$ is a compact 
subset in the Hausdorff space $\mathcal Y$, and hence it is a closed subset. It follows that $\mathcal Y-f(\mathcal X_n)$ is
a dense open subset of $\mathcal Y$. 
By \cite[2.1.15]{B2}, 
the subset of rigid points $(\mathcal Y-f(\mathcal X_n))\cap Y$ 
is open dense in $Y$. Here we provide
$Y$ with the topology 
induced from the Berkovich space $\mathcal Y$. But
this topology on $Y$ is induced 
by a complete metric. In fact, it is the unit polydisc in $K^n$ provided with the metric 
given by the valuation of $K$. By the Baire category theorem (\cite[9.1]{O}), the set $$\bigcap_{n=1}^\infty \Big((\mathcal 
Y-f(\mathcal X_n))\cap Y\Big)=(\mathcal Y-f(\mathcal X))\cap Y$$ is dense in 
$Y$. In particular, it is nonempty. 
This contradicts to the assumption that $f:X\to Y$ is surjective. 
So $f(\mathcal X)\cap \mathcal S$ is nonempty. Take $x\in \mathcal X$ 
such that $f(x)\in \mathcal S$. Then by \cite[9.1.3]{B2}, we have
$$\mathrm{dim}\,X\geq d(\mathcal H(x)/K)\geq s(\mathcal H(x)/K)\geq s(\mathcal H (f(x))/K)=\mathrm{dim}\,Y.$$


\begin{thebibliography}{22}

\bibitem{SGA4} M. Artin, A. Grothendieck and J. L. Verdier, \emph{Th\'eorie des
Topos et Cohomologie \'Etale des Sch\'emas}, Lecture Notes in Math.
269, 270, 305, Springer-Verlag (1972--1973).

\bibitem{AM} M. Atiyah and Macdonald, \emph{Introduction to Commutative Algebra}, Addison-Wesley, Reading, 
Mass. (1969). 

\bibitem{Berk} V. Berkovich, Vanishing cycles for formal schemes II,
\emph{Invent. Math.} 125 (1996), 367-390.

\bibitem{B2} V. Berkovich,  \emph{Spectral Theory and Analytic Geometry over Non-Archimedean 
Fields}, Mathematical Surveys and Monographs, vol. 33, Amercian Mathematical Society (1990). 

\bibitem{BE} S. Bloch and H. Esnault, Local Fourier transform and rigidity for
$\mathcal D$-modules, \emph{Asian J. Math.} 8 (2004), 587-606.

\bibitem{BGR} S. Bosch, U. G\"untzer and R. Remmert, \emph{Non-Archimedean Analysis. A Systematic 
Approach to Rigid Geometry}, Grundlehren der Mathematischen Wissenschaften, Bd. 261, Springer-Verlag (1984). 

\bibitem{C} B. Conrad, Generic fibers of deformations rings, available at 
http://math.stanford.edu/~conrad/modseminar/. 

\bibitem{Ddualite} P. Deligne, Dualit\'e,
in \emph{Cohomologie \'Etale} (SGA $4\frac{1}{2}$), Lecture Notes in
Math. 569, Springer-Verlag (1977), 154-167.

\bibitem{F} L. Fu,
Deformation of $\ell$-adic sheaves with undeformed local monodromy,
\emph{J. Number Theory} 133 (2013), 675-691.

\bibitem{SGA1} A. Grothendieck, \emph{Rev\^etements \'Etales et Groupe
Fondamental} (SGA 1), Lecture Notes in Math. 224, Springer-Verlag
(1971).

\bibitem{dJ} A. J. de Jong, Crystalline Dieudonn\'e module theory
via formal and rigid geometry, \emph{Publ. Math. IHES} 82 (1995),
5-96.

\bibitem{K} N. Katz, {\it Rigid Local Systems}, Annals of Math.
Studies 139, Princeton University Press (1996).


\bibitem{Kisin} M. Kisin, Moduli of finite flat group schemes and modularity, \emph{Ann. of Math. } 170
(2009), 1085-1180.

\bibitem{M} B. Mazur, Deforming Galois representations, \emph{Galois groups
over $\mathbb Q$}, 385-437, Math. Sci. Res. Inst. Publ., 16,
Springer, New York, 1989.

\bibitem{O} J. C. Oxtoby, \emph{Measure and Category},  Springer-Verlag (1971).

\bibitem{S} M. Schlessinger, Functors on Artin rings, {\it Trans.
A.M.S.} 130 (1968), 208-222.
\end{thebibliography}
\end{document}